\documentclass[a4paper,12pt, twoside, pdftex]{amsart}

\usepackage{fullpage}
\usepackage{amsmath, amsthm, amssymb}
\usepackage{txfonts}
\usepackage{amsfonts}
\usepackage{prettyref}
\usepackage{mathrsfs}
\usepackage[all]{xy}
\usepackage[utf8]{inputenc}

\usepackage{tikz}
\usepackage{dynkin-diagrams}
\usetikzlibrary{backgrounds}

\usepackage[foot]{amsaddr} 

\theoremstyle{plain}
\newtheorem{thm}{Theorem}[section]
\newtheorem{cor}[thm]{Corollary}
\newtheorem{lem}[thm]{Lemma}
\newtheorem{prop}[thm]{Proposition}
\newtheorem*{acknowledgements}{Acknowledgements}

\newtheorem*{notations}{Notations and conventions}

\theoremstyle{definition}

\theoremstyle{remark}
\newtheorem{rmk}{Remark}[section]

\numberwithin{equation}{section}

\newrefformat{th}{Theorem~\ref{#1}}
\newrefformat{cr}{Corollary~\ref{#1}}
\newrefformat{lm}{Lemma~\ref{#1}}
\newrefformat{dl}{Definition-Lemma~\ref{#1}}
\newrefformat{df}{Definition~\ref{#1}}
\newrefformat{cl}{Claim~\ref{#1}}
\newrefformat{sl}{Sublemma~\ref{#1}}
\newrefformat{pr}{Proposition~\ref{#1}}
\newrefformat{cj}{Conjecture~\ref{#1}}
\newrefformat{st}{Step~\ref{#1}}
\newrefformat{sc}{Section~\ref{#1}}
\newrefformat{df}{Definition~\ref{#1}}
\newrefformat{rm}{Remark~\ref{#1}}
\newrefformat{q}{Question~\ref{#1}}
\newrefformat{pb}{Problem~\ref{#1}}
\newrefformat{cd}{Condition~\ref{#1}}
\newrefformat{eg}{Example~\ref{#1}}
\newrefformat{he}{Heore~\ref{#1}}
\newrefformat{fg}{Figure~\ref{#1}}
\newrefformat{tb}{Table~\ref{#1}}
\newrefformat{as}{Assumption~\ref{#1}}

\newcommand{\pref}{\prettyref}
\newcommand{\Art}{\operatorname{\textbf{Art}_{\bfk}}}
\newcommand{\codim}{\operatorname{codim}}
\newcommand{\coh}{\operatorname{coh}}
\newcommand{\Coh}{\operatorname{Coh}}
\newcommand{\Coker}{\operatorname{Coker}}
\newcommand{\Cone}{\operatorname{Cone}}
\newcommand{\Def}{\operatorname{Def}}
\newcommand{\Ext}{\operatorname{Ext}}
\newcommand{\Gr}{\operatorname{Gr}}
\newcommand{\Hom}{\operatorname{Hom}}
\newcommand{\id}{\operatorname{id}}
\newcommand{\perf}{\operatorname{perf}}
\newcommand{\Pf}{\operatorname{Pf}}
\newcommand{\rank}{\operatorname{rank}}
\newcommand{\Set}{\operatorname{\textbf{Set}}}
\newcommand{\Spec}{\operatorname{Spec}}
\newcommand{\Spf}{\operatorname{Spf}}
\newcommand{\supp}{\operatorname{supp}}

\newcommand{\cE}{\mathcal{E}}
\newcommand{\cF}{\mathcal{F}}
\newcommand{\cG}{\mathcal{G}}
\newcommand{\cH}{\mathcal{H}}
\newcommand{\cI}{\mathcal{I}}
\newcommand{\cN}{\mathcal{N}}
\newcommand{\cO}{\mathcal{O}}
\newcommand{\cP}{\mathcal{P}}
\newcommand{\cT}{\mathcal{T}}
\newcommand{\cX}{\mathcal{X}}

\newcommand{\bC}{\mathbb{C}}
\newcommand{\bN}{\mathbb{N}}
\newcommand{\bP}{\mathbb{P}}
\newcommand{\bZ}{\mathbb{Z}}

\newcommand{\bfk}{\mathbf{k}}

\newcommand{\bfG}{\mathbf{G}}

\newcommand{\bfQ}{\mathbf{Q}}

\newcommand{\scrX}{\mathscr{X}}

\newcommand{\frakm}{\mathfrak{m}}

\title{Algebraic deformations and Fourier--Mukai transforms for Calabi--Yau manifolds}
\author[H.~Morimura]{Hayato Morimura}
\address{SISSA, via Bonomea 265, 34136 Trieste, Italy}

\email{
 hmorimur@sissa.it}

\date{}
\begin{document}
\maketitle

\begin{abstract}
Given a pair of derived-equivalent Calabi--Yau manifolds of dimension more than two,
we prove that the derived equivalence can be extended to
general fibers of versal deformations.
As an application,
we give a new proof of the Pfaffian--Grassmannian derived equivalence.
\end{abstract}

\section{Introduction}
Let $X_0$ be a Calabi--Yau manifold of dimension more than two in the strict sense,
i.e.,
a smooth projective variety over a field $\bfk$ with trivial canonical bundle
and
$\text{H}^i (X_0, \cO_{X_0}) = 0$ for $0 < i < \dim X_0$.
Then the deformation functor
\begin{align*}
F_{X_0} = \Def_{X_0} \colon \Art \to \Set
\end{align*}
of $X_0$ has a universal formal family $(R, \xi)$,
which is effective by \cite[Theorem III 5.4.5]{GD61}
and
there exists an effectivization $X_R$ flat and projective over $ R$,
whose formal completion along the closed fiber $X_0$ is isomorphic to $\xi$.
Since deformations of Calabi--Yau manifolds are unobstructed,
the complete local noetherian ring $R$ is regular
and
we have
\begin{align*}
R \cong \bfk \llbracket t_1, \ldots, t_d \rrbracket,
\end{align*}
where $d =\dim_{\bfk} \text{H}^1 (X_0, \cT_{X_0})$.
By \cite[Theorem 1.6]{Art69b} 
there exists a versal deformation $X_S$ flat and of finite type over $S$,
where $S$ is an algebraic $\bfk$-scheme  
with a distinguished closed point $s$
such that
the formal completion along the closed fiber $X_0$ over $s$ is isomorphic to $\xi$.
It is known that the triple $(S, s, X_S)$ is unique only locally around $s$ in the \'etale topology.
Unwinding the construction,
one finds a nonsingular affine variety $S$
over which the versal deformation $X_S$ is smooth projective.
Our main result is the following:

\begin{thm}[ Theorem 4.1 ] \label{thm:VDDeqIntro}
Let $X_0$ and $X_0^\prime$ be derived-equivalent Calabi--Yau manifolds
of dimension more than two.
Then there exists a nonsingular affine variety $S$ over $\bfk$
such that
general fibers of smooth projective versal deformations $X_S$ and $X^\prime_S$ over $S$ 
are derived-equivalent.
In particular,
after possible shrinking of the base scheme $S$,
the schemes $X_S$ and $X^\prime_S$ are derived-equivalent.
\end{thm}

The relationship between deformations and Fourier--Mukai transforms has been addressed
in \cite{Tod} for first order deformations of smooth projective varieties,
in \cite{BBP} for formal deformations of complex tori,
and
in \cite{HMS09} for formal deformations of K3 surfaces
by deforming Fourier--Mukai kernels.
In the above cases,
a relative Fourier--Mukai transform of $n$-th order deformations
induces an isomorphism
which associates
to the direction of a $(n+1)$-th order deformation of one side
that of the other side.
So the fiber product deforms along the pair of the directions
to yield the fiber product of the $(n+1)$-th order deformations.
Then it is natural to ask
whether one can deform the Fourier--Mukai kernel to a perfect complex
on the fiber product of the $(n+1)$-th order deformations,
and
the relative integral functor defined by the deformed perfect complex is an equivalence.

For Calabi--Yau manifolds of dimension more than two,
the isomorphism induced by a relative Fourier--Mukai transform
connects a pair of $(n +1)$-th order deformations of complex structures. 
Moreover,
since the effectivizations
$X_R$
and
$X^\prime_R$
are smooth over $R$,
the obstruction class to deforming a perfect complex \cite{Low, Lie}
is given by the product of
the relative Atiyah class
and 
the relative Kodaira--Spencer class \cite[Corollary 3.4]{HT10}.
In this paper,
based on the argument in \cite[Section 3]{HMS09},
we deform a Fourier--Mukai kernel
defining the derived equivalence of $X_0$ and $X^\prime_0$
along the sequence of the natural quotient maps
\begin{align*}
\cdots \to R/\frakm^{n+2}_R \to R/\frakm^{n+1}_R \to R/\frakm^n_R \to \cdots.
\end{align*}
From the compatible system of deformed Fourier--Mukai kernels,
we obtain an effectivization as a perfect complex
by \cite[Proposition 3.6.1]{Lie}.
Passing through a filtered inductive system
of finitely generated subalgebras of $R$
whose colimit is $R$,
we obtain a perfect complex on $X_S \times_S X^\prime_S$
which restricts to the Fourier--Mukai kernel on $X_0 \times X^\prime_0$    
by \cite[Proposition 2.2.1]{Lie}
and
the construction of the versal deformations.
Then the standard argument shows that
the relative integral functor defined by the perfect complex is an equivalence.
One can also show the derived equivalence of effectivizations of universal formal families,
in particular,
that of formal deformations of $X_0$ and $X^\prime_0$.

As an application,
we give a new (but slightly weaker) proof of the Pfaffian--Grassmannian equivalence,
which is conjectured in \cite{Rod},
explained in \cite{0803.2045} from a physical perspective,
and
proved in \cite{BC, 0610957, ADS}.
Via \pref{thm:VDDeqIntro},
the derived equivalence is induced by that of the complete intersections of $G_2$-Grassmannians \cite{Kuz18, Ued}.
Similarly, due to \cite[Proposition 4.7]{IIM}
the derived equivalence of the intersections of two Grassmannians in $\bP^9$ \cite{BCP}
is induced by
that of the complete intersections in $G (2, 5)$ \cite{KR, Mor}.
We expect to find a new example of Fourier--Mukai partners through deformation methods using \pref{thm:VDDeqIntro}.

\begin{notations}
We work over an algebraically closed field $\bf{k}$ of characteristic 0 throughout the paper.
For an augmented $\bf{k}$-algebra $A$,
by $\frakm_A$ we denote its augmentation ideal.
For a noetherian formal scheme $\scrX$,
by $D^b (\scrX)$ we denote the bounded derived category of the abelian category $\Coh (\scrX)$ of coherent sheaves on $\scrX$. 
\end{notations}

\begin{acknowledgements}
The author would like to express his gratitude to Kazushi Ueda for suggesting the problem.
The author would like to thank Paolo Stellari for inviting him to the university of Milan, his hospitality, and helpful discussions. 
The author also would like to thank Andrea Tobia Ricolfi for offering many corrections to earlier versions of this paper. 
The author thanks Yukinobu Toda for informing the author on the paper \cite{HMS09},
Atsushi Ito and Makoto Miura on the paper \cite{Kuh}.
\end{acknowledgements}


\section{Smooth projective versal deformations}
When it comes to deformations,
Calabi--Yau manifolds are equipped with nice geometric features.
In this section, after reviewing some basics on deformation theory of schemes,
we explain how to construct smooth projective versal deformations of Calabi--Yau manifolds
of dimension more than two. 

\subsection{Infinitesimal deformations of schemes}
Let $X$ be a $\bfk$-scheme.
A deformation of $X$ over a local artinian $\bfk$-algebra $A$ with residue field $\bfk$ is a pair $(X_A, i_A)$,
where $X_A$ is a scheme flat over $A$
and
$i_A \colon X \hookrightarrow X_A$ is a closed immersion
such that
the induced map $X \to X_A \times_A \bfk$ is an isomorphism.
Two deformations $(X_A, i_A)$ and $(Y_A, j_A)$ are said to be equivalent
if there is an $A$-isomorphism $X_A \to Y_A$ compatible with $i_A$  and $j_A$.
The deformation functor $F_X = \Def_X \colon \Art \to \Set$ sends each $A \in \Art$
to the set of equivalence classes of deformations of $X$ over $A$.

Assume that $X$ is projective over $\bfk$.
Then $F_X$ satisfies Schlessinger's criterion \cite{Sch}
and there exists a miniversal formal family $(R, \xi)$ for $F_X$ \cite[Theorem 18.1]{Har10},
where $R$ is a complete local noetherian $\bfk$-algebra with residue field $\bfk$,
and
$\xi$ belongs to the limit
\begin{align*}
\hat{F}_X (R) = \displaystyle \lim_{\longleftarrow} F_X (R / \frakm^n_R)
\end{align*}
of the inverse system
\begin{align*}
\cdots \to F_X (R/\frakm_R^{n+2}) \to F_X (R/\frakm_R^{n+1}) \to F_X (R/\frakm_R^n) \to \cdots
\end{align*}
induced by the natural quotient maps $R/\frakm^{n+1}_R \to R/\frakm^n_R$.
The formal family $\xi$ corresponds to a natural transformation 
\begin{align} \label{eq:Cfun}
h_R = \Hom_{\operatorname{\bfk-alg}} (R, -) \to F_X,
\end{align}
which sends each homomorphism $f \in h_R (A)$ factorizing through
$R \to R / \frakm^{n+1}_R \xrightarrow{g} A$
to
$F_X (g) (\xi_n)$ \cite[Proposition 15.1]{Har10}.
The functor \pref{eq:Cfun} is strongly surjective by versality of $\xi$.
So for every surjection $B \to A$ in $\Art$ the map
\begin{align*}
h_R (B) \to h_R (A) \times_{F_X (A)} F_X (B)
\end{align*}
is surjective.
In particular, the map $h_R (A) \to F_X (A)$ is surjective  for each $A \in \Art$. 

Let $X_n$ be the schemes which define $\xi_n$.
Then by \cite[Proposition 21.1]{Har10} there is a noetherian formal scheme $\scrX$ over $R$
such that
$X_n \cong \scrX \times_R R / \frakm^{n+1}_R$ for each $n$.
By abuse of notation,
we use the same symbol $\xi$ to denote the formal scheme $\scrX$.
Thus any scheme which defines an equivalence class $[X_A, i_A]$ can be obtained as the pullback of $\xi$
along some morphism of noetherian formal schemes $\Spec A \to \Spf R$.
If $X$ is regular,
then the Zariski tangent space of $\Spec R$ at the closed point is $\text{H}^1 (X, \cT_X)$.
Assume further that $\text{H}^0 (X, \cT_X) = 0$,
i.e.,
the scheme $X$ has no infinitesimal automorphisms which restrict to the identity of $X$.
Then every equivalence class $[X_A, i_A]$ is just a deformation $(X_A, i_A)$
and
we have $h_R \simeq F_X$ \cite[Corollary 18.3]{Har10}.
In this case, the functor $F_X$ is said to be pro-representable
and
$(R, \xi)$ a universal formal family for $F_X$.

\subsection{Algebraization}
Towards an algebraic family of deformations of $X$,
the first step is to find a  scheme $X_R$ flat and of finite type over $R$
whose formal completion along the closed fiber $X$ is isomorphic to $\xi$.
If $X$ is projective and $\text{H}^2 (X, \cO_X) = 0$,
i.e.,
deformations of any invertible sheaf on $X$ are unobstructed,
then by \cite[Theorem III5.4.5]{GD61}
there exists such a scheme $X_R$.
In this case, the formal family $(R,\xi)$ is said to be effective.
One sees that
the scheme $X_R$ appeared in the proof of \cite[Theorem III5.4.5]{GD61}  is projective over $R$.
We will call such $X_R$ an effectivization of $\xi$.

The next step is to find an algebraic $\bfk$-scheme $S$ with a distinguished closed point $s \in S$,
and
a scheme $X_S$ flat and of finite type over $S$
whose formal completion along the closed fiber $X$ over $s$ is isomorphic to $\xi$.
The deformation functor $F_X$ can naturally be extended to a functor defined on the category
$\operatorname{\textbf{k-alg}^\text{aug}}$
of augmented noetherian $\bfk$-algebras.
By abuse of notation,
we use the same symbol $F_X$ to denote the extended functor,
which sends each $(B, \frakm_B) \in \operatorname{\textbf{k-alg}^\text{aug}}$ to the set of equivalence classes of deformations over $(B, \frakm_B)$.
The following is the well-known fact necessary to prove the existence of such a triple $(S, s, X_S)$.
We provide a proof as we could not find any reference in the literature.

\begin{lem} \label{lem:LFP}
Let $X$ be an algebraic $\bfk$-scheme.
Then the functor
$F_X$ 
is locally of finite presentation,
i.e.,
for every filtered inductive system of augmented noetherian $\bfk$-algebras
$\left\{ \left( B_i, \frakm_{B_i} \right) \right\}_{i \in I}$
whose colimit is $B$,
the canonical map
\begin{align*}
\displaystyle \lim_{\longrightarrow} F_X \left( \left( B_i, \frakm_{B_i} \right) \right) \to F_X \left( \left( B, \frakm_B \right) \right)
\end{align*}
is bijective.
\end{lem}
\begin{proof}
To show the surjectivity,
let $[ X_B, i_B ]$ be an element in $F_X \left( \left( B, \frakm_B \right) \right)$.
By \cite[Corollary IV11.2.7]{GD66} 
for some index $\lambda \in I$
there exists a scheme $X_{B_\lambda}$ flat and of finite type over $B_\lambda$
with a $B$-isomorphism
$X_B \to X_{B_\lambda} \times_{B_\lambda} B$.
Then an element
$\left\{ [ X_{B_k}, i_{B_k} ] \right\}_{k \geq \lambda} \in \displaystyle \lim_{\longrightarrow} F_X \left( \left( B_i, \frakm_{B_i} \right) \right)$
is sent to
$[ X_B, i_B ]$
by the canonical map.

To show the injectivity,
let $\left\{ [ X_{B_k}, i_{B_k} ] \right\}_{k \geq j}$
and
$\left\{ [ Y_{B_k}, j_{B_k} ] \right\}_{k \geq j}$
be two elements sent to the same equivalence class
$[ X_B, i_B ] = [ Y_B, j_B ]$. 
By \cite[Theorem IV8.8.2, Corollary IV8.8.2.4]{GD66}
for some index $l \geq j$
there is a $B_l$-isomorphism $X_{B_l} \to {Y_{B_l}}$ sent to the $B$-isomorphism $X_B \to Y_B$.
Since we have $[ X_B, i_B ] = [ Y_B, j_B ]$,
the isomorphism is compatible with $i_{B_l}$ and $j_{B_l}$.
Thus
$\left\{ [ X_{B_k}, i_{B_k} ] \right\}_{k \geq l}$
and
$\left\{ [ Y_{B_k}, j_{B_k} ] \right\}_{k \geq l}$
define the same element in
$\displaystyle \lim_{\longrightarrow} F_X \left( \left( B_i, \frakm_{B_i} \right) \right)$.
\end{proof}

By \pref{lem:LFP}
one can apply \cite[Theorem 1.6]{Art69b} to obtain such a triple $(S, s, X_S)$.
The scheme $X_S$ is said to be a versal deformation over $S$
and
the miniversal formal family $(R, \xi)$ is said to be algebraizable.
Since some details are necessary in the sequel,
we show the existence of a versal deformation
when $X$ is a higher dimensional Calabi--Yau manifold.
 
\begin{thm} \label{thm:Alg}
Let $X_0$ be a Calabi--Yau manifold of dimension more than two.
Then every effective universal formal family $(R, \xi)$ for $F_{X_0}$ is algebraizable.
\end{thm}
\begin{proof}
Let $T = \Spec \bfk [ t_1, \ldots, t_d ]$
and
$t \in T$ be the closed point corresponding to a maximal ideal $(t_1, \ldots, t_d)$.
Since the formal completion of $\cO_T (T)$ along $(t_1, \ldots, t_d)$ is isomorphic to $R$,
there is a filtered inductive system $\{ R_i \}_{i \in I}$
of finitely generated $\cO_T (T)$-subalgebras of $R$
whose colimit is $R$.
Choose a finite type presentation
\begin{align*}
R_i = \cO_T (T) [Y] / \left( f \left (Y \right) \right),
\end{align*}
where $Y = (Y_1, \ldots, Y_N)$ and $f = (f_1, \ldots, f_m)$.
Then we have the solution
$\hat{y} = (\hat{y}_1, \ldots, \hat{y}_N)$
of
$f(Y ) = 0$
in $R$ corresponding to the canonical homomorphism $R_i \to R$ \cite[Corollary 1.6]{Art69a}.
Since $F_{X_0}$ is locally of finite presentation,
$[ X_R, i_R ]$ is the image of some element
$\zeta_i \in F_{X_0} \left( \left( R_i, \frakm_{R_i} \right) \right)$
by the canonical map
$F_{X_0} \left( \left( R_i, \frakm_{R_i} \right) \right) \to F_{X_0} (R)$.
By \cite[Corollary 2.1]{Art69a} there exist an \'etale neighborhood $S$ of $t$ in $T$,
and a solution
$y = (y_1, \ldots, y_N)$
in
$\cO_S (S)$
with
\begin{align} \label{eq:SOL}
y_i \equiv \hat{y_i} \ \ ( \bmod \ \frakm^2_R),
\end{align}
i.e., $y$ and $\hat{y}$ induce the same element in $F_{X_0} ( R / \frakm^2_R )$.
Let $\varphi \colon R_i \to \cO_S (S)$ be the homomorphism corresponding to the solution $y$,
and let $[ X_S, i_S ]$ be the image of $\zeta_i$ by the map $F_{X_0} (\varphi)$
and
$\{ \eta_n \}_{n \in \bN}$ the formal family induced by $[ X_S, i_S ]$.
From \pref{eq:SOL} it follows
\begin{align*}
F_{X_0} (\psi_1) ([ X_R, i_R ]) = \xi_1 = \eta_1,
\end{align*}
where $\psi_1 \colon R \to R / \frakm^2_R$ is the natural surjection. 
By versality of $(R, \xi)$
there is a compatible sequence of homomorphisms
$\psi_n \colon R \to R / \frakm^{n+1}_R$
lifting $\psi_{n - 1}$
and
such that
$F_{X_0} (\psi_n) ([ X_R, i_R ]) = \eta_n$
for every positive integer $n$.
The sequence $\{ \psi_n \}_{n \in \bN}$ induces a homomorphism $\psi \colon R \to R$
such that
\begin{align*}
F_{X_0} (\psi) ([ X_R, i_R ]) \equiv \eta_n \ \ ( \bmod \ \frakm^{n+1}_R).
\end{align*}
Since $\psi$ is the identity modulo $\frakm^2_R$,
it is an automorphism. 
Thus the formal completion of $X_S$ along the closed fiber $X_0$ is isomorphic to $\xi$. 
\end{proof}

\subsection{Smoothness and projectivity}
If $S^\prime$ is an \'etale neighborhood of $s$ in $S$,
then the scheme $X_{S^\prime}$ obtained in the same way gives another versal deformation.
The following lemma is crucial for the rest of the paper.

\begin{lem} \label{lem:SFP}
Let $X_0$ be a Calabi--Yau manifold of dimension more than two.
Then there exists a nonsingular affine variety $S$ over $\bfk$
with a versal deformation $X_S$
which is  projective
and
smooth of relative dimension $\dim X_0$ over $S$.
\end{lem}
\begin{proof}
An \'etale neighborhood of $t$ in $T$ is smooth over $\bfk$.
Since an open immersion is \'etale,
we may assume that $S$ is connected.
Then $S$ must be irreducible,
otherwise the local ring $\cO_{S,s}$ has more than one minimal prime ideal
for every point $s$ in the intersection of irreducible components.
We already know that $X_S$ is flat over $S$.
Since $X_R$ is projective over $R$,
by \cite[Theorem  IV8.10.5]{GD66} there exists an index $j$
such that
for all $k \geq j$ the schemes $X_{R_k}$ are projective over $R_k$.
A base change of projective morphism is projective \cite[Tag 02V6]{SP}. 

Since $S$ is irreducible
and
$\pi_S \colon X_S \to S$ is flat and proper,
the restriction of $\pi_S$ to each irreducible component of $X_S$ is surjective.
In particular, each irreducible component contains the closed fiber $X_0$
and
we have
\begin{align*}
\operatorname{rel.dim (\pi_S)} = \dim X_0.
\end{align*}
Note that the function
\begin{align*}
n_{X_S / S} \colon S \to \bZ_{\geq 0} \cup \{ \infty \},
\end{align*}
which sends every point $s \in S$ to the dimension of the fiber over $s$ is locally constant,
since $\pi_S$ is flat and proper \cite[Tag 0D4J]{SP}.
Again, we have used the irreducibility of $S$. 

Due to \pref{lem:sm} below,
the morphism $\pi_R \colon X_R \to \Spec R$ is smooth.
We claim that there is an index $l$
such that
for all $k \geq l$ the morphisms $X_{R_k} \to \Spec R_k$ are smooth.  
To show this,
we may assume that $X_{R_i}$ are affine.
Let $R \to B$ be the ring homomorphism corresponding to $\pi_R$.
Then there exists a finitely generated $\bZ$-subalgebra $R_0$ of $R$
and
a smooth ring homomorphism $R_0 \to B_0$
such that
$B \cong B_0 \otimes_{R_0} R$ \cite[Tag 00TP]{SP}.
By \cite[Tag 07C3]{SP} the inclusion $R_0 \to R$ factors through $R_l$ for some index $l$.
Since smoothness is stable under base change,
the claim follows.
\end{proof}

\begin{lem} \label{lem:sm}
The scheme $X_R$ is regular
and
the morphism $\pi_R \colon X_R \to \Spec R$ is smooth of relative dimension $\dim X_0$.
\end{lem}
\begin{proof}
We adapt the proof of \cite[Proposition III10.4]{Har}.
Since the base scheme $\Spec R$ has the only one closed point,
every closed point $x \in X_R$ belongs to the closed fiber $X_0$.
By \cite[Tag 031E]{SP} the local ring $\cO_{X_R, x}$ is regular,
so $X_R$ is regular.
Note that $X_R$ is irreducible.
From the proof of \cite[Tag 031E]{SP} one sees that $\pi_R$ induces the injection
$\frakm_R / \frakm^2_R \to \frakm_{\cO_{X_R, x}} / \frakm^2_{\cO_{X_R, x}}$,
which is dual to the surjection
\begin{align*}
T_{\pi_R} \colon T_x X_R \to T_{\pi_R (x)} R
\end{align*} 
of Zariski tangent spaces.
It follows that
\begin{align*}
\dim_{\bfk (x)} \left( \Omega_{X_R / R} \otimes \bfk (x) \right) = \dim X_0
\end{align*}
for every closed point $x \in X_R$.
Since $\pi_R$ is flat and of finite type, we also have
\begin{align*}
\dim_{\bfk (\zeta)} \left( \Omega_{X_R / R} \otimes \bfk (\zeta) \right) = \dim X_0
\end{align*}
for the generic point $\zeta$ of $X_R$ \cite[Theorem I4.8A, Theorem II8.6A]{Har}. 
Then by \cite[Lemma II8.9]{Har} the coherent sheaf $\Omega_{X_R / R}$ is locally free of rank $\dim X_0$.
\end{proof}

\begin{rmk}
From the proof of \pref{lem:SFP},
one sees that
the scheme $X_S$ must be irreducible,
otherwise $X_0$ becomes disconnected.
\end{rmk}
\section{Deformations of Fourier--Mukai kernels}  
In order to define a relative integral functor from $D^b (X_S)$ to $D^b (X^\prime_S)$,
we deform the Fourier--Mukai kernel $\cP_0$  to a perfect complex $\cP_S$
on the fiber product $X_S \times_S X^\prime_S$ of smooth projective versal deformations. 
Although applying the main theorem in \cite{Lie} might suffice,
we adopt a more concrete approach based on \cite{HMS09, HT10},
still using some results from \cite{Lie}.
Here,
for a deformation $X_B$ of a $\bfk$-scheme $X$
over an augmented noetherian $\bfk$-algebra $(B, \frakm_B)$,
by a deformation of a perfect complex $E$ on $X$ over $(B, \frakm_B)$
we mean a pair $(E_B, u_B)$,
where $E_B \in D^b(X_B)$
and
$u_B \colon E_B \otimes^L_B \bfk \to E$ is an isomorphism.

\subsection{Derived equivalence and relative Hochschild cohomology}
Let $X_n$ and $X^\prime_n$ be derived-equivalent schemes
smooth and projective over $R_n$.
The schemes $X_n$, $X^\prime_n$, and their fiber product $X_n \times_{R_n} X^\prime_n$ over $R_n$ form the following diagram
\begin{align*}
\begin{gathered}
\xymatrix{
&
X_n \times_{R_n} X^\prime_n
\ar _{ q_n }[dl]
\ar ^{ p_n }[dr]
&\\
X_{R_n}
&
&
X^\prime_{R_n}}
\end{gathered}
\end{align*}
with the natural projections $q_n$ and $p_n$.
For any perfect complex $\cP_n$ on $X_n \times_{R_n} X^\prime_n$,
the relative integral functor
\begin{align*}
\Phi_{\cP_n} \left( - \right)  = Rp_{n *} \left( \cP_n \otimes^L q^*_n \left( - \right) \right)
\end{align*}
sends each object in $D^b(X_n)$ to $D^b(X^\prime_n)$.
Due to the Grothendieck--Verdier duality
the functor $\Phi_{\cP_n}$ has the right adjoint,
which we denote by
$\Phi_{(\cP_n)_R}$.

Assume that $\Phi_{\cP_n}$ is an equivalence.
Then two functors
\begin{align*}
\begin{gathered}
\Psi_1 \colon D^b(X_n \times_{R_n} X_n) \to D^b(X_n \times_{R_n} X^\prime_n) \\
\cG \mapsto \cP_n * \cG, 
\end{gathered}
\end{align*}
\begin{align*}
\begin{gathered}
\Psi_2 \colon D^b(X_n \times_{R_n} X^\prime_n) \to D^b(X^\prime_n \times_{R_n} X^\prime_n) \\
\cG^\prime \mapsto \cG^\prime * (\cP_n)_R 
\end{gathered}
\end{align*} 
respectively induce isomorphisms
\begin{align*}
\begin{gathered}
\psi_1 \colon \Ext_{X_n \times_{R_n} X_n} (-, -) \to \Ext_{X_n \times_{R_n} X^\prime_n} \left(\Psi_1 \left( - \right), \Psi_1 \left( - \right) \right),
\end{gathered}
\end{align*}
\begin{align*}
\begin{gathered}
\psi_2 \colon \Ext_{X_n \times_{R_n} X^\prime_n} \left( - , - \right) \to \Ext_{X^\prime_n \times_{R_n} X^\prime_n} \left(\Psi_2 \left( - \right), \Psi_2 \left( - \right) \right).
\end{gathered}
\end{align*} 
The composition defines the isomorphism
\begin{align*}
\Phi^{\text{HH}^*}_{\cP_n} = \psi_2 \circ \psi_1\colon \text{HH}^
* (X_n / R_n) \to \text{HH}^* (X^\prime_n / R_n)
\end{align*}
of the relative Hochschild cohomology complex \cite{Cal},
which gives rise to the isomorphism
\begin{align*}
\Phi^{\text{HT}^*}_{\cP_n} = ( \text{I}_{X^\prime_n}^\text{HKR} )^{-1} \circ \Phi^{\text{HH}^*}_{\cP_n} \circ \text{I}_{{X_n}}^\text{HKR} \colon \text{HT}^* (X_n / R_n) \to \text{HT}^* (X^\prime_n / R_n),
\end{align*}
where 
\begin{align*}
\text{I}_{X_n}^\text{HKR} \colon \text{HT}^* (X_n / R_n) \to \text{HH}^* (X_n / R_n),
\end{align*}
\begin{align*}
\text{I}_{X^\prime_n}^\text{HKR} \colon \text{HT}^* (X^\prime_n / R_n) \to \text{HH}^* (X^\prime_n / R_n)
\end{align*}
are the relative Hochschild--Kostant--Rosenberg isomorphisms.
Namely, we have the following commutative diagram
\begin{align} \label{eq:CD0}
\begin{gathered}
\xymatrix{
\text{HT}^*(X_n/R_n) \ar^{\text{I}_{X_n}^\text{HKR}}[r] \ar_{\Phi^{\text{HT}^*}_{\cP_n}}[d]
& \Ext^*_{X_n \times_{R_n} X_n}(\cO_{\Delta_n}, \cO_{\Delta_n}) \ar^{\Phi^{\text{HH}^*}_{\cP_n} = \psi_2 \circ \psi_1}[d] \\
\text{HT}^* (X^\prime_n / R_n) 
& \Ext^*_{X^\prime_n \times_{R_n} X^\prime_n}(\cO_{\Delta^\prime_n}, \cO_{\Delta^\prime_n}). \ar_{\left( \text{I}_{X^\prime_n}^\text{HKR} \right)^{-1}}[l]
}
\end{gathered}
\end{align}
\subsection{Relative Atiyah class and HKR isomorphism}
For a perfect complex $E_n$ on $X_n$
the relative Atiyah class is the element
\begin{align*}
A(E_n) \in \Ext^1_{X_n} (E_n, E_n \otimes \Omega_{\pi_n})
\end{align*}
induced by the boundary morphism of the short exact sequence 
\begin{align*}
0
\to I_n / I^2_n
\to \cO_{X_n \times_{R_n} X_n } / I_n
\to \cO_{\Delta_n}
\to 0,
\end{align*}
where $I_n$ is the defining ideal sheaf of the relative diagonal. 
Composition in $D^b(X_n)$ and exterior product
$\Omega^{\otimes i}_{\pi_n} \to \Omega^i_{\pi_n}$
yield the  exponential
\begin{align*}
\exp \left( A(E_n) \right) \in \bigoplus \Ext^i_{X_n} (E_n, E_n \otimes \Omega^i_{\pi_n}).
\end{align*}

Consider the automorphism
$\tau \colon X_n \times_{R_n} X_n \to X_n \times_{R_n} X_n$
interchanging the two factors. 
Since the conormal bundle $I_n / I^2_n$
consists of elements of the form
$x \otimes_{R_n} 1 - 1 \otimes_{R_n} x$,
the pullback $\tau^*$ acts on
$\text{H}^p (X_n, \wedge^q \cT_{\pi_n})$
by $(-1)^q$.
Then as a straightforward generalization of \cite[Lemma 5.8]{Tod},
one obtains two commutative diagrams
\begin{align} \label{eq:CD1}
\begin{gathered}
\xymatrix{
\text{HT}^*(X_n/R_n) \ar^{\tau^* \circ \text{I}_{X_n}^\text{HKR}}[r] \ar_{q^*_n}[d]
& \Ext^*_{X_n \times_{R_n} X_n}(\cO_{\Delta_n}, \cO_{\Delta_n}) \ar^{\psi_1}[d] \\
\text{HT}^*(X_n \times_{R_n} X^\prime_n/R_n) \ar^{\exp \left( A \left( \cP_n \right) \right) \cdot}[r]
& \Ext^*_{X_n \times_{R_n} X^\prime_n}(\cP_n, \cP_n),
}
\end{gathered}
\end{align} 
\begin{align} \label{eq:CD2}
\begin{gathered}
\xymatrix{
\text{HT}^*(X^\prime_n /R_n) \ar^{\text{I}^{\prime \text{HKR}}_{X^\prime_n}}[r] \ar_{p^*_n}[d]
& \Ext^*_{X^\prime_n \times_{R_n} X^\prime_n}(\cO_{\Delta^\prime_n}, \cO_{\Delta^\prime_n}) \ar^{\psi^{-1}_2}[d] \\
\text{HT}^*(X_n \times_{R_n} X^\prime_n/R_n) \ar^{\exp \left( A \left( \cP_n \right)\right) \cdot}[r]
& \Ext^*_{X_n \times_{R_n} X^\prime_n}(\cP_n, \cP_n).
}
\end{gathered}
\end{align}

\subsection{Obstruction class}
There exists an obstruction for a perfect complex
$E_n$ on $X_n$
to deform to some perfect complex on $X_{n + 1}$ \cite{Low, Lie}.
By \cite[Corollary 3.4]{HT10}
it has the explicit expression as the product 
of the truncated Atiyah class of $E_n$
and
the truncated Kodaira--Spencer class of the thickenings
$X_n \hookrightarrow X_{n + 1}$
defined by a square zero ideal.
In our setting,
the deformation $X_{n+1}$ is smooth over $R_{n+1}$.
So the truncated Atiyah and Kodaira--Spencer classes coincide with the relative Atiyah and Kodaira--Spencer classes respectively.
Then the obstruction class is given by
\begin{align*}
\varpi (E_n) = (\id_{E_n} \otimes \ \kappa_n) \circ A(E_n)
\in \Ext^2_{X_n} (E_n, E_n^{\oplus l_n}),
\end{align*}
where $\kappa_n \in \Ext^1_{X_n}(\Omega_{\pi_n}, \cO_{X_n}^{\oplus l_n})$ denotes the relative Kodaira--Spencer class,
which is the extension class  
of the short exact sequence
\begin{align*}
0
\to \cO_{X_n}^{\oplus l_n}
\xrightarrow{\cdot {}^t\!( ds_1 \cdots ds_{l_n})} \Omega_{X_{n + 1}} |_{X_n}
\to \Omega_{\pi_n}
\to 0.
\end{align*}
Here
$l_n = \dim_\bfk \frakm^{n + 1}_R / \frakm^{n + 2}_R$
and
$\{ s_1, \ldots, s_{l_n} \}$ is a fixed basis of the $\bfk$-vector space
$\frakm^{n + 1}_R / \frakm^{n + 2}_R$.  

Suppose that
there exists a thickening
$X^\prime_n \hookrightarrow X^\prime_{n + 1}$
whose relative Kodaira--Spencer class is
$\kappa^\prime_n = \left( \Phi^{\text{HT}^2}_{\cP_n} \right)^{\oplus l_n} (\kappa_n)
\in \text{H}^1 (\cT_{\pi^\prime_n})^{\oplus l_n}$.
Let
$\kappa_n \boxplus \kappa^\prime_n = q^*_n \kappa_n + p^*_n \kappa^\prime_n
\in \text{H}^1 (\cT_{\pi_n \times \pi^\prime_n})^{\oplus l_n}$.
Adapting \cite[Lemma 3.7]{HMS09} to our setting in a straightforward way,
we obtain the following:

\begin{lem} \label{lem:OBSf}
Under the above assumption
there exists a perfect complex $\cP_{n+1}$ on
$X_{n+1} \times_{R_{n+1}} X^\prime_{n+1}$
with an isomorphism
$\cP_{n+1} \otimes^L_{R_{n+1}} R_n \cong \cP_n$
such that
the integral functor
$\Phi_{\cP_{n+1}} \colon D^b(X_{n+1}) \to D^b(X^\prime_{n+1})$
is an equivalence.
\end{lem}
\begin{proof}
We show the vanishing of the obstruction class
\begin{align*}
A(\cP_n) \cdot (\kappa_n \boxplus \kappa^\prime_n)
\in \Ext^2_{X_n \times_{R_n} X^\prime_n} (\cP_n, \cP_n)^{\oplus l_n}.
\end{align*}
We write $\kappa_n, \kappa^\prime_n$ as
$\kappa_n = ( \kappa^1_n, \ldots, \kappa^{l_n}_n), 
\kappa^\prime_n = ( \kappa^{\prime 1}_n, \ldots, \kappa^{\prime l_n}_n )$
with respect to the fixed basis.
By commutativity of the diagrams \pref{eq:CD0}, \pref{eq:CD1}, and \pref{eq:CD2},
we have 
\begin{align*}
A(\cP_n) &= A(\cP_n) \cdot (\kappa^i_n \boxplus \kappa^{\prime i}_n) \\
&= \psi_1 \left( \tau^* \left( I^{\text{HKR}}_{X_n} \left( \kappa^i_n \right) \right) \right) + \psi^{-1}_2 \left( I^{\prime \text{HKR}}_{X^\prime_n} \left( \kappa^{\prime i}_n \right) \right) \\
&= \psi^{-1}_2 I^{\prime \text{HKR}}_{X^\prime_n} \left( - \left( I^{\prime \text{HKR}}_{X^\prime_n} \right) ^{-1} \left( \psi_2 \left( \psi_1 \left( I^{\text{HKR}}_{X_n} \left( \kappa^i_n \right)  \right) \right) \right)  + \kappa^{\prime i}_n \right) \\
&= \psi^{-1}_2 I^{\prime \text{HKR}}_{X^\prime_n} \left( -\Phi^{\text{HT}^2}_{\cP_n} \left( \kappa^i_n \right) + \kappa^{\prime i}_n \right)
= 0
\end{align*}
for each $i$.
Note that,
as mentioned above,
the pullback $\tau^*$ acts on
$\text{H}^1 (X_n \times_{R_n} X^\prime_n, \cT_{\pi_n \times \pi^\prime_n})$
by $-1$.
So there exists a deformation
$\cP_{n+1}$
of
$\cP_n$.
Then by \cite[Proposition 1.3]{LST} the functor $\Phi_{\cP_{n+1}}$ is an equivalence,
since $\cP_0$ defines an equivalence
and
$\pi_n, \pi^\prime_n$ are smooth projective.
\end{proof}


Combining \pref{lem:OBSf} and \pref{lem:DD} below,
one sees that
if the closed fibers $X_0$ and $X^\prime_0$ are higher dimensional Calabi--Yau manifolds,
then one can always deform a Fourier--Mukai kernel on
$X_0 \times X^\prime_0$   
to
some Fourier--Mukai kernel on
$X_n \times_{R_n} X^\prime_n$
for arbitrary order $n$.   

\begin{lem} \label{lem:DD}
If the closed fiber of $X^\prime_0$ is a Calabi--Yau manifold of dimension more than two,
then there exists a thickening
$X^\prime_n \hookrightarrow X^\prime_{n + 1}$
whose relative Kodaira--Spencer class is
$\kappa^\prime_n = \left( \Phi^{\text{HT}^2}_{\cP_n} \right)^{\oplus l_n} (\kappa_n)$.
\end{lem}
\begin{proof}
First, we show the vanishing of cohomology
$H^0 (\wedge^2 \cT_{\pi^\prime_n})$
and
$H^2 (\cO_{X^\prime_n})$.
Since $\pi^\prime_n$ is a projective morphism of noetherian schemes
and
the sheaves
$\wedge^2 \cT_{\pi^\prime_n}$,
$\cO_{X^\prime_n}$
are flat over $R_n$, 
by \cite[Theorem III12.8]{Har}
there is a Zariski open neighborhood $U \subset \Spec R_n$ of the closed point
such that
$\dim_{k(y)} H^0 (\wedge^2 \cT_{\pi^\prime_n, y }) = 0$
and
$\dim_{k(y)} H^0 (\cO_{X^\prime_n, y}) = 0$
for all $y \in U$.
Then we have $U = \Spec R_n$,
as the complement does not contain the only closed point of $\Spec R_n$.

Next, we construct the thickening
$X^\prime_n \hookrightarrow X^\prime_{n + 1}$.
Fix an affine open covering $\{ U_i \}$ of $X^\prime_n$.
The element $\kappa^\prime_n \in \text{H}^1(\cT_{\pi^\prime_n})^{\oplus l_n}$
is represented by $2$-cocycles $\{ \theta_{ij} \}$ with respect to $\{ U_i \}$. 
By \cite[Propsotion 3.6, Exercise 5.2]{Har10}
these cocycles define automorphisms of the trivial deformations
$U_{ij} \times_{R_n} \Spec R_{n+1}$
that can be glued to make a global deformation
$X^\prime_{n+1}$ of $X^\prime_n$.
By definition,
the relative Kodaira--Spencer class of this thickening is $\kappa^\prime_n$.
\end{proof}

\subsection{Algebraization}
Now, we have a system of deformations
$\cP_n \in \perf (X_{R_n} \times_{R_n} X^\prime_{R_n})$ of $\cP_0$
with compatible isomorphisms
$\cP_{n + 1} \otimes^L_{R_{n + 1}} R_n \to \cP_n$. 
By \cite[Proposition 3.6.1]{Lie}
there exists an effectivization,
i.e.,
a perfect complex $\cP_R$ on $X_R \times_R X^\prime_R$
with compatible isomorphisms
$\cP_R \otimes^L_R R_n \to \cP_n$. 
Recall that
in Section $2$
to algebrize $X_R$
we have used a filtered inductive system $\{ R_i \}_{i \in I}$
of finitely generated $\cO_T (T)$-subalgebras of $R$
whose colimit is $R$.
For a sufficiently large index $i$,
there are deformations $X_{R_i}, X^\prime_{R_i}$ of $X_0, X^\prime_0$ over $R_i$
whose pullback along the canonical homomorphism $R_i \to R$ are $X_R, X^\prime_R$.
So we have
\begin{align*}
X_R \times_R X^\prime_R \cong \left( X_{R_i} \times_{R_i} X^\prime_{R_i} \right) \times_{R_i} R.
\end{align*}
By \cite[Proposition 2.2.1]{Lie}
there exists a perfect complex
$\cP_{R_i}$ on $X_{R_i} \times_{R_i} X^\prime_{R_i}$
with an isomorphism
$\cP_{R_i} \otimes^L_{R_i} R \to \cP_R$.
Then the derived pullback
$\cP_S \in \perf (X_S \times_S X^\prime_S)$
along $R_i \to \cO_S (S)$
is a deformation of $\cP_0$.
Finally,
we obtain the following:

\begin{prop}
Let $\cP_0$ be a Fourier--Mukai kernel 
defining the derived equivalence of Calabi--Yau manifolds $X_0$ and $ X^\prime_0$ of dimension more than two.
Then there exists a perfect complex $\cP_S$ on the fiber product
$X_S \times_S X^\prime_S$
of smooth projective versal deformations
with an isomorphism
$\cP_S \otimes^L_{\cO_S(S)} \bfk \to \cP_0$.
\end{prop}

\begin{rmk}
For fixed universal formal families $(R, \xi)$ and  $(R, \xi^\prime)$ of $X_0$ and $X^\prime_0$,
the deformations $X^\prime_n$ 
associated with the images
$\kappa^\prime_n = \left( \Phi^{\text{HT}^2}_{\cP_n} \right)^{\oplus l_n} (\kappa_n)$
determine another universal formal family $(R, \tilde{\xi}^\prime)$ of $X^\prime_0$.
Since two universal formal families $(R, \xi^\prime)$ and $(R, \tilde{\xi}^\prime)$ are isomorphic up to unique isomorphism,
by the construction of our versal deformations
we may algebrize $X_R$ and $X^\prime_R$ simultaneously.
\end{rmk}
\section{Proof of the main theorem}  
The smooth projective versal deformations $X_S$, $X^\prime_S$, and their fiber product $X_S \times_S X^\prime_S$ over $S$ form the following diagram
\begin{align*} 
\begin{gathered}
\xymatrix{
&
X_S \times_S X^\prime_S
\ar _{ q }[dl]
\ar ^{ p}[dr]
&\\
X_S
&
&
X^\prime_S}
\end{gathered}
\end{align*}
with the natural projections $q$ and $p$.
The relative integral functor
\begin{align*}
\Phi_{\cP_S} \left( - \right)  = Rp_* \left( \cP_S \otimes^L q^* \left( - \right) \right)
\end{align*}
sends each object in $D^b (X_S)$ to $D^b (X^\prime_S)$,
where
$\cP_S \in \perf (X_S \times_S X^\prime_S)$
is a deformation of $\cP_0$ over $\left( \cO_S (S), \frakm_{\cO_S (S)} \right)$.
In this section,
after possible shrinking of the base scheme $S$,
we show that the functor $\Phi_{\cP_S}$ is an equivalence. 
One can also show similar results for formal deformations and their effectivizations.

\begin{thm} \label{thm:VDDeq}
Let $X_0$ and $X_0^\prime$ be derived-equivalent Calabi--Yau manifolds
of dimension more than two.
Then there exists a nonsingular affine variety $S$ over $\bfk$
such that
general fibers of smooth projective versal deformations $X_S$ and $X^\prime_S$ over $S$ 
are derived-equivalent.
In particular,
after possible shrinking of the base scheme $S$,
the schemes $X_S$ and $X^\prime_S$ are derived-equivalent.
\end{thm}
\begin{proof}
Due to the Grothendieck--Verdier duality
the functor $\Phi_{\cP_S}$ has the left adjoint,
which we denote by $\Phi_{(\cP_S)_L}$.
For every object $E \in D^b (X_S)$ the counit morphism
$\eta \colon \Phi_{(\cP_S)_L} \circ \Phi_{\cP_S} \to \id_{D^b (X_S)}$
gives a distinguished triangle
\begin{align} \label{eq:DT}
\Phi_{(\cP_S)_L} \circ \Phi_{\cP_S} (E) \to E \to F \coloneqq \Cone \left( \eta \left( E \right) \right).
\end{align}
We may assume that $E$ and $F$ are perfect complexes on $X_S$.
Let
$ i_s \colon X_s \hookrightarrow X_S$,
$ i^\prime_s \colon X^\prime_s \hookrightarrow X^\prime_S$,
and
$ j_s = i_s \times i^\prime_s \colon X_s \times X^\prime_s \hookrightarrow X_S \times_S X^\prime_S$
be the closed immersions for every closed point $s \in S$.
By the derived flat base change we have
\begin{align*}
L{ i^\prime_s}^* \Phi_{\cP_S} \left( E \right) \cong \Phi_s \left( E_s \right),
\end{align*}
where $E_s =  L i^*_s (E)$
and
$\Phi_s = \Phi_{j^*_s \cP_S}$.
We also denote by $\left( \Phi_s \right)_L$ the left adjoint of $\Phi_s$ with kernel $\left( j^*_s \cP_S \right)_L$.
Then \pref{eq:DT} restricts to a distinguished triangle
\begin{align*}
\left( \Phi_s \right)_L \circ \Phi_s (E_s) \to E_s \to F_s.
\end{align*}
Note that the restriction of the counit morphism is the counit morphism.
Since $\Phi_{\cP_0}$ is an equivalence,
the restriction of $F$ to $X_0$ is quasi-isomorphic to $0$.
So the support $\supp (F) = \bigcup \supp \cH ^l (F)$ of the perfect complex $F$
is a proper Zariski closed subset of $X_S$.
Let $U \subset S$ be the complement of the image $\pi_S \left( \supp (F) \right)$.
Since $U$ contains the image of the closed fiber $X_0$,
it is a nonempty open subset of $S$
and
$\pi_S^{-1} (U)$ does not intersect with $\supp (F)$.
In particular,
we have $F_s \cong 0$
for every closed point $s \in U$.
If $E$ is a strong generator of $D^b(X_S)$,
this implies that $\Phi_{\cP_U}$ is fully faithful.
Here $\cP_U$ denotes the restriction of $\cP_S$ to $q^{-1} \pi^{-1}_S (U)$.
Recall that
a triangulated category is strongly finitely generated
if there exist an object $E$ and nonnegative integer $k$
such that
every object can be obtained from $E$
by taking
isomorphisms,
direct summands,
shifts,
and
not more than $k$ times cones.
Since $X_S$ is noetherian, separated, and regular,
$D^b (X_S)$ is strongly finitely generated by \cite[Theorem 3.1.4]{BB}.
Since $\Phi_{\cP_S}$ and $\Phi_{(\cP_S)_L}$ commute with direct sums on $D^b(X_S)$ by \cite[Corollary 3.3.4]{BB},
we may assume that $E$ has no nontrivial direct summands.
Using the cohomology long exact sequence induced by a distinguished triangle,
one inductively sees that 
on $\pi^{-1}_S (U)$ the cone of the counit morphism for any object is quasi-isomorphic to $0$. 
Similarly, one finds a Zariski open subset $V \subset S$
such that
$\left( \Phi_{\cP_V} \right)_L$ is fully faithful.
Finally, we obtain an equivalence $\Phi_s \colon D^b (X_s) \to D^b (X^\prime_s)$
for every closed point $s \in U \cap V \neq \emptyset$,
as a fully faithful functor which admits either fully faithful left or right adjoint is an equivalence.
In particular, 
$\Phi_{\cP_S}$ is an equivalence
after possible shrinking of the base scheme $S$.
\end{proof}

\begin{cor} \label{cor:GFdeq}
Let $X_0$ and $X_0^\prime$ be derived-equivalent Calabi--Yau manifolds of dimension more than two.
Then all effectivizations $X_R$ and $X^\prime_R$ of universal formal families
$\xi$ and $\xi^\prime$
projective over $R$ are derived-equivalent.
\end{cor}
\begin{proof}
Replace $X_S, X^\prime_S$, and $S$ by $X_R, X^\prime_R$, and $R$
in the above proof.
Then $\pi_R (\supp F)$ is a Zariski closed subset of $\Spec R$
which does not contain the only one closed point. 
This implies that $\supp F$ is empty.
It remains to show that the derived equivalence does not depend on the choice of $X_R$ and $X^\prime_R$. 
Given an effectivization
$\pi_R \colon X_R \to \Spec R$
of $\xi$,
we have the following pullback diagram
\begin{align*}
\begin{gathered}
\xymatrix{
\xi \cong \hat{X}_R \ar@{^{(}->}[r]^-{} \ar_{\hat{\pi}}[d]
& X_R \ar^{\pi_R}[d] \\
\Spf R \ar@{^{(}->}[r]_-{}
& \Spec R
}
\end{gathered}
\end{align*}
of noetherian formal schemes,
where $X_R$ is considered as the formal completion along itself.
Since $\pi_R$ is projective
and
$R$ is a complete local noetherian ring,
by \cite[Corollary III5.1.6]{GD61} the functor
\begin{align*} 
\coh (X_R) \to \Coh (\hat{X}_R),
\end{align*}
which sends each coherent sheaf $\cF$ on $X_R$ to its formal completion $\hat{\cF}$ along the closed fiber
is an equivalence of abelian categories.
So we have
\begin{align*} 
D^b (X_R) \simeq D^b (\hat{X}_R).
\end{align*}
In particular, all effectivizations of $\xi$ are derived-equivalent. 
\end{proof}

\begin{cor} \label{cor:FDdeq}
Let $X_0$ and $X_0^\prime$ be derived-equivalent Calabi--Yau manifolds of dimension more than two.
Then for any formal deformation
$\cX = \hat{X}_{\bfk \llbracket t \rrbracket}$ of $X_0$
there exists a formal deformation
$\cX ^\prime = \hat{X^\prime}_{\bfk \llbracket t \rrbracket}$ of $X^\prime_0$
which is derived-equivalent to $\cX$.
\end{cor}
\begin{proof}
From the argument in Section $3$ and the above proof,
it follows immediately.
\end{proof}
\section{Pfaffian--Grassmannian equivalence}
By \pref{thm:VDDeq} the derived equivalence of central fibers of versal deformations
can be extended to that of general fibers.
After studying deformations of the relevant Calabi--Yau $3$-folds,  
we show that
the Pfaffian--Grassmannian derived equivalence
is induced by
the derived equivalence of IMOU varieties.

\subsection{Grassmannian side}
Let $\cE$ be a locally free sheaf on a smooth projective variety $Z$ over $\bC$
and
let $Y_0$ be the zero scheme of a section $s \in H^0 (Z, \cE)$
with $\codim Y_0 = \rank \cE$
and
the defining ideal sheaf $\cI_{Y_0} \subset \cO_Z$.
By \cite{Weh} a sufficient condition for every algebraic deformation of $Y_0$ to be obtained by varying the section in $ H^0 (Z, \cE)$ is the vanishing of cohomology
\begin{align*}
H^1 (Z, \cE \otimes \cI_{Y_0}), \ \ H^1 (Y_0, \cT_Z |_{Y_0}).
\end{align*}
Now,
we will consider the case 
\begin{align*}
Y_0 = \Gr (2,V_7)_{1^7} \coloneqq \Gr (2,V_7) \cap \bP (W),
\end{align*}
where $V_7$ is a $7$-dimensional complex vector space
and
$W$ is a $14$-dimensional general quotient vector space of $\wedge^2 V_7 \twoheadrightarrow W$.

\begin{lem} \label{lem:GrDef}
Every deformation of $Y_0 = \Gr (2,V_7)_{1^7}$ can be obtained by varying the section $s$.
\end{lem}
\begin{proof}
It suffices to show the vanishing of cohomology
\begin{align*}
H^1 (\Gr (2,V_7), \cO_{\Gr (2,V_7)} (1)^{\oplus 7} \otimes \cI_{Y_0}), \ \ H^1 (Y_0, \cT_{\Gr (2,V_7)} |_{Y_0}).
\end{align*}
By \cite[(1.4)]{Kuh} we have two spectral sequences
\begin{align*}
\text{H}^p ( \Gr (2, V_7), \cF \otimes \wedge^{q+1} \cO (-1)_{\Gr (2,V_7)} ) \ \
&\Rightarrow
\ \ \text{H}^{p-q} \left( \Gr (2, V_7), \cF \otimes \cI_{Y_0} \right) \text , \ q \geq 0, \\
\text{H}^p \left( \Gr (2, V_7), \cF \otimes \wedge^q \cO (-1)_{\Gr (2,V_7)} \right) \ \
&\Rightarrow
\ \ \text{H}^{p-q} \left( \Gr (2, V_7), \cF |_{Y_0} \right)
\end{align*}
for any locally free sheaf $\cF$ on $\Gr (2, V_7)$.
Then Borel--Bott--Weil theorem gives the desired result.
\end{proof}

\subsection{Pfaffian side}
Let
$H_{Y^\prime_0} \colon \Art \to \Set$
be the functor of embedded deformations of a projective scheme $Y^\prime_0 \subset \bP^6$ over $\bC$.
Then we have the forgetful functor $H_{Y^\prime_0} \to F_{Y^\prime_0}$.
Let
$\text{t}_{H_{Y^\prime_0}} \to \text{t}_{F_{Y^\prime_0}}$
be the induced map of tangent spaces,
which is given by
\begin{align*}
H_{Y^\prime_0} ( \bC [t] / t^2 ) \to F_{Y^\prime_0} ( \bC [t] / t^2 ).
\end{align*}
Now,
we will consider the case 
\begin{align*}
Y^\prime_0 = \Pf (4,V_7) \cap \bP (W^\perp) \cong \Pf (4,V_7) \cap \bP^6,
\end{align*}
where $W^\perp = \Coker (W^\vee \hookrightarrow \wedge^2 V^\vee_7)$.

\begin{lem} \label{lem:TSsurj}
The induced map of tangent spaces
\begin{align*}
 \operatorname{t}_{H_{Y^\prime_0}} \to \operatorname{t}_{F_{Y^\prime_0}}
\end{align*}
is surjective.
\end{lem}
\begin{proof}
We have an exact sequence
\begin{align} \label{eq:Pfseq}
0 \to \cO_{\bP^6}(-7) \to 7\cO_{\bP^6}(-4) \to 7\cO_{\bP^6}(-3) \to \cO_{\bP^6} \to \cO_{Y^\prime_0} \to 0.
\end{align}
From the cohomology of \pref{eq:Pfseq} and the restriction of Euler sequence
we obtain $\text{H}^1 (\cT_{\bP^6} |_{Y^\prime_0}) \cong 0$.
Since ${Y^\prime_0}$ is nonsingular,
the short exact sequence
\begin{align*}
0 \to \cT_{Y^\prime_0} \to \cT_{\bP^6} |_{Y^\prime_0} \to \cN_{Y^\prime_0 / \bP^6} \to 0
\end{align*}
gives rise to a long exact sequence of cohomology
\begin{align*}
0 &\to \text{H}^0 (\cT_{Y^\prime_0}) \to \text{H}^0 (\cT_{\bP^6} |_{Y^\prime_0}) \to \text{H}^0 (\cN_{{Y^\prime_0} / \bP^6}) \\
&\xrightarrow{\delta^0} \text{H}^1 (\cT_{Y^\prime_0}) \to \text{H}^1 (\cT_{\bP^6} |_{Y^\prime_0}) \to \text{H}^1 (\cN_{{Y^\prime_0} / \bP^6}) \to \cdots,
\end{align*}
where the boundary map $\delta^0$ coincides with $ \operatorname{t}_{H_{Y^\prime_0}} \to \operatorname{t}_{F_{Y^\prime_0}}$ by \cite[Proposition 20.2]{Har10}.
\end{proof}

\begin{lem} \label{lem:PffDef}
Every deformation of $Y^\prime_0 = \Pf (4,V_7) \cap \bP^6$ lifts to an embedded deformation in $\bP^6$.
\end{lem}
\begin{proof}
It is well-known that $H_{Y^\prime_0}$ is pro-representable and unobstructed.
We also know that $F_{Y^\prime_0}$ is pro-representable.
Then \pref{lem:TSsurj} allows us to apply \cite[Exercise 15.8]{Har10}
and the forgetful functor $H_{Y^\prime_0} \to F_{Y^\prime_0}$ is strongly surjective.
In particular, it is surjective. 
\end{proof}

\subsection{Induced derived equivalence}
Let $X_0$ be the complete intersection in $G_2$-Grassmannian
$\bfG = G_2 / P$
associated with the crossed Dynkin diagram
$\dynkin{G}{*x}$
defined by an equivariant vector bundle
$\cE_{(1,1)} \cong G_2 \times_P V^P_{(1,1)}$,
which is a flat degeneration of $Y_0$ \cite[Proposition 5.1]{IIM}.
Let $X^\prime_0$ be the complete intersection in $G_2$-Grassmannian
$\bfQ = G_2 / Q$
associated with the crossed Dynkin diagram
$\dynkin{G}{x*}$
defined by an equivariant vector bundle
$\cF_{(1,1)} \cong G_2 \times_Q V^Q_{(1,1)}$,
which is a flat degeneration of $Y^\prime_0$ \cite[Theorem 7.1]{KK}.
It is known that
the Calabi--Yau $3$-folds $X_0$ and $X^\prime_0$
are derived-equivalent \cite{Kuz18, Ued}.

\begin{cor}
The Calabi--Yau $3$-folds $Y_0 = \Gr (2, V_7)_{1^7}$ and \ $Y^\prime_0 = \Pf (4,V_7) \cap \bP^6$
are derived-equivalent.
\end{cor}
\begin{proof}
By definition of $Y_0$ and \pref{lem:GrDef}
general fibers of a versal deformation of $X_0$ are isomorphic to $\Gr (2, V_7)_{1^7}$.
By \cite[Corollary 6.3]{KK} and \pref{lem:PffDef}
general fibers of a versal deformation of $X^\prime_0$ are isomorphic to $\Pf (4,V_7) \cap \bP^6$.
Then
$\Gr (2, V_7)_{1^7}$
and
$\Pf (4,V_7) \cap \bP^6$
are derived-equivalent by \pref{thm:VDDeq}.
\end{proof}

\begin{rmk}
The derived-equivalent pair obtained here does not carry any information about $W$ and $W^\perp$,
while the original Pfaffian--Grassmannian equivalence connects
$\Gr (2,V_7) \cap \bP (W)$
with
$\Pf (4,V_7) \cap \bP (W^\perp)$ for every $W$.
We have proved that
for a generic choice of $W$ the $Y_0$
is derived-equivalent to
the $Y^\prime_0$ associated with some other $W$. 
\end{rmk}



\begin{thebibliography}{999999}
\bibitem[ADS15]{ADS}
N. Addington, W. Donovan, and E. Segal, \emph{The Pfaffian--Grassmannian equivalence revisited}, Algebr.\ Geom. 2(3), 332-364 (2015). 

\bibitem[Art69a]{Art69a}
M. Artin, \emph{Algebraic approximation of structures over complete local rings}, Publ.\  Math-Paris. 36, 23-58 (1969).

\bibitem[Art69b]{Art69b}
M. Artin, \emph{Algebraization of formal moduli: I}, Global Analysis (Papers in Honor of K. Kodaira), University of Tokyo Press, 21-71 1969,  ISBN:978-1-4008-7123-0.

\bibitem[BB03]{BB}
A. Bondal and M. Van den Bergh, \emph{Generators and representability of functors in commutative and noncommutative geometry}, Mosc. Math. J. 3(1), 1-36 (2003).

\bibitem[BBP07]{BBP}
O. Ben-Bassat, J. Block, and T. Pantev, \emph{Non-commutative tori and Fourier-Mukai duality}, Compos.\ Math. 143(2), 423-475 (2007). 

\bibitem[BC09]{BC}
L. Borisov and A. C\u{a}ld\u{a}raru, \emph{The Pfaffian--Grassmannian derived equivalence}, J.\ Algebraic.\ Geom. 18(2), 201-222 (2009).

\bibitem[BCP20]{BCP}
L. Borisov, A. C\u{a}ld\u{a}raru, and A. Perry, \emph{Intersections of two Grassmannians in $\bP^9$}, J.\  Reine.\  Angew.\  Math. 2020(760), 133-162 (2020). 

\bibitem[C\u{a}l10]{Cal}
A. C\u{a}ld\u{a}raru, \emph{The Mukai pairing, I: a categorical approach}, New York J.\ Math. 16, 61-98 (2010). 

\bibitem[GD61]{GD61}
A. Grothendieck and J. Dieudonn\'e, \emph{Elements de g\'eom\'etrie algebriqu\'e: III. \'Etude cohomologique des faisceaux coh\'erent, première partie.}, Publ.\  Math-Paris. 11, 5-167 (1961).

\bibitem[GD66]{GD66}
A. Grothendieck and J. Dieudonn\'e, \emph{Elements de g\'eom\'etrie algebriqu\'e: IV. \'Etude locale des sch\'emas et des morphismes de sch\'emas, Troisième partie.}, Publ.\  Math-Paris. 28, 5-255 (1966).

\bibitem[Har77]{Har}
R. Hartshorne, \emph{Algebraic geometry}, Graduate Texts in Mathematics, 52, Springer-Verlag, 1977, ISBN: 0-387-90244-9.

\bibitem[Har10]{Har10}
R. Hartshorne, \emph{Deformation theory}, Graduate Texts in Mathematics, 257, Springer-Verlag, 2010, ISBN: 978-1-4419-1596-2.

\bibitem[HMS09]{HMS09}
D. Huybrechts, E. Macr\`i, and P. Stellari, \emph{Derived equivalences of K3 surfaces and orientation}, Duke Math.\ J. 149(3), 461-507 (2009).

\bibitem[HT07]{0803.2045}
K. Hori and D. Tong, \emph{Aspects of non-abelian gauge dynamics in two-dimensional $\cN = (2,2)$}, J.\ High Energy Phys. 79(5), 41pp (2007).

\bibitem[HT10]{HT10}
D. Huybrechts, R. Thomas, \emph{Deformation-obstruction theory for complexes via Atiyah and Kodaira–Spencer classes}. Math.\ Ann. 346, 545 (2010).

\bibitem[IIM19]{IIM}
D. Inoue, A. Ito, and M. Miura, \emph{Complete intersection Calabi--Yau manifolds with respect to homogeneous vector bundles on Grassmannians}, Math.\ Z. 292(1-2), 677-703 (2019). 

\bibitem[KK16]{KK}
G. Kapustka and M. Kapustka, \emph{Calabi--Yau threefolds in $\bP^6$}, Ann.\ Mat.\ Pur.\ Appl. 195(2), 529-556 (2016).  

\bibitem[KR19]{KR}
M. Kapustka and M. Rampazzo, Marco, \emph{Torelli problem for Calabi--Yau threefolds with GLSM description}, Commun.\ Num.\ Theory.\ Phys. 13(4), 725-761 (2019).

\bibitem[K\"uc96]{Kuh}
O. K\"uchle, \emph{Some properties of Fano manifolds that are zeros of sections in homogeneous vector bundles over Grassmannians}, Pac.\ J.\ Math. 175(1), 117-125 (1996). 

\bibitem[Kuz]{0610957}
A. Kuznetsov, \emph{Homological projective duality for Grassmannians of lines}, arXiv:math/0610957

\bibitem[Kuz18]{Kuz18}
A. Kuznetsov, \emph{Derived equivalence of {Ito--Miura--Okawa--Ueda} {Calabi--Yau} $3$-folds}, J.\ Math.\ Soc.\ Jap. 70(3), 1007-1013 (2018).

\bibitem[Lie06]{Lie}
M. Lieblich, \emph{Moduli of complexes on a proper morphism}, J. \ Algebraic Geom. 15, 175-206 (2006).

\bibitem[Low05]{Low}
W. Lowen, \emph{Obstruction theory for objects in abelian and derived categories}, Comm.\ Algebra. 33, 3195-3223 (2005).

\bibitem[LST13]{LST}
A.C. L\'opez Mart\'in, D. S\'anchez G\'omez, and C. Tejero Prieto, \emph{Relative Fourier--Mukai transforms for Weierstrass fibrations, abelian schemes and Fano fibrations}, Math.\ Proc.\ Cambridge. 155(1), 129-153 (2013).

\bibitem[Mor21]{Mor}
H. Morimura, \emph{Derived equivalences for the flops of type $C^2$ and $A^4_G$ via mutation of semiorthogonal decomposition}, Algebras Represent. Theory, https://doi.org/10.1007/s10468-021-10036-y

\bibitem[R{\o}d00]{Rod}
A. R{\o}dland, \emph{The Pfaffian Calabi--Yau, its mirror, and their link to the Grassmannian $G(2,7)$}, Compos.\ Math. 122(2), 135-149 (2000). 

\bibitem[Sch68]{Sch}
M. Schlessinger, \emph{Functors of Artin rings}, Trans.\ Amer.\ Math.\ Soc. 130, 208-222 (1968). 

\bibitem[SP]{SP}
The Stacks Project authors, \emph{Stacks project},
https://stacks.math.columbia.edu.

\bibitem[Ued19]{Ued}
K. Ueda, \emph{$G_2$-Grassmannians and derived equivalences}, Manuscripta Math. 159(3-4), 549-559 (2019).

\bibitem[Tod09]{Tod}
Y. Toda, \emph{Deformations and Fourier--Mukai transforms}, J.\ Differential.\ Geom. 81(1), 197-224 (2009). 

\bibitem[Weh84]{Weh}
J. Weher, \emph{Deformation of varieties defined by sections in homogeneous vector bundles}, Math.\ Ann. 268(4), 519-532 (1984). 
\end{thebibliography}
\end{document}